\documentclass[a4paper,11pt]{amsart}

\usepackage{epigraph}
\usepackage{amsfonts}
 \usepackage{enumerate}
 \usepackage{amsmath, amsthm, amscd, amssymb}
\usepackage{pdfsync}
\usepackage[title]{appendix}
\usepackage[all]{xy}
\usepackage{latexsym,graphicx}
\usepackage{xspace}
\usepackage{array}
\usepackage{newclude}
\usepackage{hyperref}
\usepackage{url}
\usepackage{enumitem}
\hypersetup{
  colorlinks,
  citecolor=black,
  linkcolor=black,
  urlcolor=black}

  \usepackage[utf8]{inputenc}
\usepackage[T1]{fontenc}

\renewcommand*{\HyperDestNameFilter}[1]{\jobname-#1} 
\numberwithin{equation}{section}
\usepackage[centering, marginparwidth=3cm]{geometry}
\usepackage{marginnote}

\addtocontents{toc}{\setcounter{tocdepth}{1}}

\usepackage[nameinlink]{cleveref}

\usepackage{tikz-cd}
\usetikzlibrary{cd}
\usepackage{leftidx}
\usetikzlibrary{positioning}
\usepackage{dsfont}
\usepackage{tikz}
\usetikzlibrary{matrix,arrows,decorations.pathmorphing,positioning}

\newcommand\blfootnote[1]{%
\begingroup
\renewcommand\thefootnote{}\footnote{#1}%
\addtocounter{footnote}{-1}%
\endgroup
}
  
  \newcommand{\Addresses}{{
  \bigskip
  \footnotesize

\textsc{CNRS, IMJ-PRG, Sorbonne Universit\'{e}, 4 place Jussieu, 75005 Paris, France}\par\nopagebreak
  \textit{E-mail address}, G.~Baldi: \texttt{baldi@imj-prg.fr} 

}}

\theoremstyle{plain}
\newtheorem{theor}{Theorem}[section]

\newtheorem{conj}[theor]{Conjecture}

\newtheorem{defi}[theor]{Definition}

\newtheorem{prop}[theor]{Proposition}
\newtheorem{cor}[theor]{Corollary}
\theoremstyle{definition}

\newtheorem{rmk}[theor]{Remark}
\newtheorem{question}[theor]{Question}
\theoremstyle{remark}
\numberwithin{equation}{subsection}

\newcommand{\RR}{\mathbb{R}}
\newcommand{\CC}{\mathbb{C}}
\newcommand{\NN}{\mathbb{N}}
\newcommand{\ZZ}{\mathbb{Z}}
\newcommand{\QQ}{\mathbb{Q}}

\newcommand{\HH}{{\mathbf H}}
\newcommand{\LL}{{\mathbf L}}

\newcommand{\an}{\textnormal{an}}
\newcommand{\deff}{\textnormal{def}}

\newcommand{\ad}{\textnormal{ad}}
\newcommand{\der}{\textnormal{der}}

\newcommand{\MT}{\mathbf{MT}}

\DeclareMathOperator{\End}{End}
\DeclareMathOperator{\Hom}{Hom}

\DeclareMathOperator{\Gl}{GL}
\DeclareMathOperator{\Sl}{SL}

\newcommand{\Gm}{\mathbb{G}_{\textnormal{m}}}

\newcommand{\codim}{\operatorname{codim}}

\DeclareMathOperator{\NL}{NL}

\newcommand{\VV}{{\mathbb V}}

\newcommand{\HL}{\textnormal{HL}}
\newcommand{\atyp}{\textnormal{atyp}}

\newcommand{\typ}{\textnormal{typ}}
\newcommand{\pos}{\textnormal{pos}}

\newcommand{\N}{\mathbb{N}}

\newcommand{\prim}{\textnormal{prim}}

\newcommand{\Z}{\mathbb{Z}}
\newcommand{\Q}{\mathbb{Q}}

\newcommand{\R}{\mathbb{R}}

\newcommand{\Oo}{\mathcal{O}}

\newcommand{\F}{\mathcal{F}}

\usepackage[normalem]{ulem}

\newcommand{\Hh}{\mathbb{H}}

\newcommand{\C}{\mathbb{C}}

\newcommand{\Qbar}{\overline{\mathbb{Q}}}

\newcommand{\GL}{\mathbf{GL}}

\newcommand{\G}{{\mathbf G}}

\newcommand{\fg}{\mathfrak{g}}

\newcommand{\cA}{\mathcal{A}}

\usepackage{microtype}

\begin{document}

\title{Hodge theory and o-minimality at CIRM}\date{\today}  
\author{Gregorio Baldi} 

\begin{abstract}
We discuss the relationship between o-minimality and the so called Zilber-Pink conjecture. Since the work of Pila and Zannier, algebraization theorems in o-minimal geometry had profound impacts in Diophantine geometry (most notably on the study of special points in abelian and Shimura varieties). We will first focus on functional transcendence, discussing various recent and spectacular Ax-Schanuel theorems, and the related geometric part of Zilber-Pink. Armed with these tools, we will study the distribution of the Hodge locus of an arbitrary variation of Hodge structures (the typical/atypical dichotomy) and present some recent applications. We will conclude by describing the algebraicity and quasiprojectivity of images of period maps.
\end{abstract}

\maketitle

\epigraph{\emph{In memory of Nicola Gatti (1990--2023), whose boundless curiosity and passion for schemes would have made him a treasured reader.}}

\tableofcontents

\newpage

\section*{Preface}

\subsection{Hodge theory}
Compact Kähler manifolds are unique in that they are equipped with several compatible structures: complex, Riemannian, and symplectic. This rich interplay makes them fundamental objects in mathematics. Classical Hodge theory \cite{zbMATH03076397} enhances the classical topological invariants of the underlying space of a Kähler manifold 
$X$ by incorporating a $\C^*$-action. Building on de Rham’s work, the key insight is that for any given Riemannian metric on $X$, every cohomology class has a canonical \textbf{harmonic} representative. Since complex differential forms on can be uniquely expressed as sums of forms of type $(p, q)$, the fact that the $(p,q)$ components of a harmonic form remain harmonic gives rise to the celebrated Hodge decomposition. This decomposition splits the cohomology of $X$ with complex coefficients into a direct sum of complex vector spaces:
 \begin{displaymath}
\qquad \qquad  \qquad \qquad   H^n(X,\C)= \bigoplus _{p+q=n} H^{p,q} (X)  \qquad \qquad (\textbf{HoDec}).
 \end{displaymath}
This decomposition is elegantly encoded by a $\C^*$-action, where $H^{p,q}$ is the subspace on which $z\in \C^*$ acts as $z^p\overline{z}^q$. The symmetry under complex conjugation (namely $\overline{H^{p,q}}=H^{q,p}$) reflects the fact that $\C^*$ is viewed as a real algebraic torus (and usually denoted by $\mathbb{S}$).

Algebraic Geometry, on the other hand, studies algebraic varieties over complex numbers--the solution sets of algebraic equations with complex coefficients. The Kodaira embedding theorem characterizes smooth complex projective varieties among all compact Kähler manifolds and underpins the concept of a polarized Hodge structure, which grants favorable properties to the period domains of such objects that play a central role in Griffiths' approach \cite{Griffiths, MR2918237} to Hodge theory. Another key distinction between algebraic and general Kähler manifolds is that while general Kähler manifolds may not admit complex submanifolds, algebraic varieties possess them in abundance. For instance, an algebraic subvariety $Z\subset X$ of codimension $j$ (or more generally, a cycle) gives rise to an \textbf{algebraic} cohomology class $[Z]\in H^{2k} (X,\Q)$, whose image in $H^{2k} (X,\C)$ lies in $H^{k,k}(X)$ relatively to (HoDec). A rational cohomology class lying in some $H^{k,k}(X)$ is called a \textbf{Hodge class} (equivalently those are classes fixed by the action of $\C^*$). This suggests the tempting belief that the Hodge structure of a compact complex variety is governed by the geometry of its subvarieties, or more precisely, its Chow group.  A special case of this belief is indeed one of the Clay millennium problems:
\begin{conj}[Hodge]
A class $\lambda \in H^{*}(X,\Q)$ is algebraic if and only if it is Hodge.
\end{conj}

This conjecture remains a central open problem in mathematics, lying at the intersection of Complex and Algebraic Geometry. It provides a simple criterion to determine which topological cycles of a smooth projective algebraic variety can be represented by algebraic cycles. Finally Deligne \cite{MR0498551} later vastly generalized Hodge’s results, showing that the cohomology of any complex algebraic variety  (not necessarily compact) is functorially endowed with a structure called a mixed $\Z$-Hodge structure.

\subsection{Cycles and Hodge classes in families: the Hodge locus}
Recent years have seen significant activity in understanding how algebraic cycles vary within families, following a path laid in the previous century by figures such as Max Noether, Lefschetz, Weil, Griffiths, Grothendieck, and Deligne, among others. For a smooth projective family of algebraic varieties $f:\mathcal{X}\to S$, the Hodge structures associated with each fiber $\mathcal{X}_s:= f^{-1}(s)$ enrich the local system $R^*f_* \Z_{\operatorname{prim}}$ of primitive cohomology with the structure of 
$\VV$, a \textbf{polarized variation of Hodge structures} (VHS). Since Hodge theory fundamentally compares two distinct algebraic structures in a transcendental fashion, understanding how algebraic cycles and Hodge classes vary with $s\in S(\C)$ is a challenging problem. This has led to the study of the Hodge locus, a subset of the base where the dimension (over $\Q$) of the cohomology group $H^{k,k}(\mathcal{X}_s,\Q)$, or possibly some tensor construction thereof, is larger than expected.

In 1995, Cattani-Deligne-Kaplan \cite{CDK} achieved a breakthrough in Hodge theory by proving that the aforementioned Hodge locus of $(S,\VV)$ is a countable union of \textbf{algebraic} subvarieties of $S$. More recently, this result has been established using methods from o-minimal geometry \cite{BKT}. The Cattani-Deligne-Kaplan result is regarded by many as the most compelling evidence for the Hodge conjecture’s validity, alongside Deligne’s proof that Hodge classes on abelian varieties are absolute and various special cases such as the Lefschetz (1,1) theorem, and results by Schoen and Markman on Hodge-Weil 4-folds. The Hodge locus also plays a central role in conjectures like André-Oort and Zilber-Pink, which will be discussed in more detail later.
 
 A simple yet rich example, still far from being fully understood, involves hypersurfaces. A hypersurface $X$
 of degree $d$ and dimension $n$ is defined as the zero set of a homogeneous polynomial $F$ of degree 
$d$ in $n+2$ variables. It is smooth if the partial derivatives of 
$F$ do not vanish simultaneously. 

For fixed $n$ and $d$, the parameter space of hypersurfaces in $\mathbb{CP}^{n+1}$ of degree $d$ is simply the complex vector space of homogeneous polynomials of degree $d$, denoted  by $V=\C[X_0, \dots, X_{n+1}]_d$. However, the parameter space of \textbf{smooth} hypersurfaces $U_{n,d}=V - \Delta$  is more intricate. The relevant VHS enriches the (primitive) cohomology $H^n(X,\Z)$, and the Hodge locus remains far from understood. In the case $n=2$ and $d \geq 4$, this locus contains the \textbf{Noether-Lefschetz locus}, which can be defined purely using algebraic geometry:

\begin{displaymath}
\NL_d:=\{[X]\in U_{2,d} : \operatorname{Pic}(\mathbb{CP}^3)\to
\operatorname{Pic}(X) \text{  is not an isomorphism} \}. 
\end{displaymath}
For a surface $X$ outside $ \NL_d$,
every curve on $X$ has the pleasant and useful property that it is the complete intersection of $X$
with another surface in $\mathbb{CP}^{3}$. This object has been the subject of many beautiful studies by Griffiths, Green, Voisin, Ciliberto, Harris, Miranda and others.

The most pressing challenges in the field revolve around two fundamental questions:

\begin{question}\label{mainquestion}
What can we say about about Hodge loci? What do Hodge loci reveal to us?
\end{question}

\subsection{Organization of the Lecture Notes}
This is the second part of a lecture series on Hodge theory and o-minimality. The key object of study is the \emph{period map} associated with the aforementioned variations of Hodge structures (VHS, from now on). The first part \cite{ben}, delivered by B. Bakker, covered the following topics:
\begin{enumerate}
    \item Basic definitions and local definability of maps;
    \item Definable structures on period spaces;
    \item Global definability of period maps and the algebraicity of Hodge loci.
\end{enumerate}

The second part, consisting of four 1-hour lectures, discusses:
\begin{itemize}
    \item[(a)] Functional transcendence and the Ax-Schanuel theorem for period maps;
    \item[(b)] Distribution of the Hodge locus;
    \item[(c)] Algebraicity and quasi-projectivity of images of period maps.
\end{itemize}

We will build on and use the language of \cite{ben}, and refer to it whenever possible. Item (c) is a natural continuation of  \cite{ben}, whereas (a) and (b) were not covered in \emph{op. cit.}.

There are already many excellent expository notes on this subject, so we have tried to focus on topics not already addressed in the following recent references:
\begin{itemize}
    \item Bakker, Tsimerman, \emph{Lectures on the Ax-Schanuel conjecture} \cite{zbMATH07300480};
    \item Fresán, \emph{Hodge theory and o-minimality (after Bakker, Brunebarbe, Klingler, Tsimerman)} \cite{Fresan2020};
    \item Tsimerman, \emph{Functional transcendence and arithmetic applications} \cite{zbMATH07250485};
    \item Klingler, \emph{Hodge theory between algebraicity and functional transcendence} \cite{zbMATH07823061}.
\end{itemize}

In particular:
\begin{itemize}
    \item Our exposition on functional transcendence follows the most recent advances (albeit this approach no longer relies on o-minimality), with original sources in \cite{2021arXiv210203384B, 2022arXiv220805182B};
    \item We compile here most of the known results concerning the Zilber-Pink conjecture for variations of Hodge structures;
    \item Each section includes several exercises of varying difficulty, which are collected in the final Section 5. References to Section 5.n.m indicate the m-th exercise in the n-th lecture.
\end{itemize}
Finally, the notes cover roughly four hours of material and we tried to stick to the program exposed in Luminy. In particular the exercises play a fundamental role to complete the discussions and provide many complementary proofs that were not exposed during the lectures.

\begin{rmk}
To conclude, we wish to mention also the beautiful monographs written by Zannier \cite{zbMATH06022187} and Pila \cite{zbMATH07516542}. Both discuss related results and focus mainly on the arithmetic side, rather the setting of VHS. Nevertheless, they are highly recommended to anyone who wants to learn more about unlikely intersections!
\end{rmk}

\subsection{Acknowledgements}
These notes were prepared for a mini-course held during a Research School as part of the CIRM-SMF program (15–19 January, Luminy, France, 2024). We thank the organizers and participants for the opportunity to share this material and for their valuable feedbacks. During the final stages of this work, the author was partially supported by the grant ANR-HoLoDiRibey of the Agence Nationale de la Recherche.

\newpage

\section{Lecture 1: (Historical) Overview}\label{lect1}

The goal of this section is to describe various instances of the Zilber-Pink philosophy (about the finiteness of \emph{atypical intersections}) and the so-called \emph{Pila-Zannier strategy}. Historically, this strategy introduced the combination of o-minimality and functional transcendence that will underpin most of our lectures \cite{zbMATH05292756, MR2800724}. For additional insights, we refer to Scanlon’s exposition on o-minimality as an approach to the André-Oort conjecture \cite{zbMATH07754034}.

As a roadmap, one might conceptualize ``three stages'' in the theory of atypical (also known as \emph{unlikely}) intersections, each of increasing complexity: (1) complex abelian varieties $A$ or, more generally, semi-abelian varieties, (2) Shimura varieties $\text{Sh}_K(G,X)$, and (3) variations of Hodge structures $(S, \VV)$. O-minimality plays a role in two critical ways: first, in the arithmetic aspects developed in (1) and extended to (2), and second, in providing a framework where the period map $S \to \Gamma \backslash D$ can naturally exist.

\begin{rmk}
The most general setting, which will not be discussed here, involves \emph{admissible graded-polarizable integral variations of mixed Hodge structures}. For details, see \cite[Sec.~4]{2024arXiv240616628B}. The Zilber-Pink conjecture stated in \emph{op. cit.} indeed implies the other cases discussed herein.
\end{rmk}

\subsection{Some Number Theory}

In this section we first introduce the \emph{Zilber-Pink conjecture} (ZP, henceforth) for the moduli space of principally polarized abelian varieties of dimension $g$, $\mathcal{A}_g$, which is the period space of weight one Hodge structures on the standard symplectic lattice \cite{ben}. ZP is a far-reaching generalization of the \emph{André-Oort} (AO) and \emph{Mordell-Lang} conjectures, which we now recall. A special case of AO can be summarized as follows \cite{zbMATH01251645}:

\begin{theor}[André]\label{thmao}
Let $V$ be an irreducible algebraic curve in the complex affine plane, which is neither horizontal nor vertical. Then $V$ is a modular curve $Y_0(N)$ for some $N > 0$ if and only if it contains infinitely many points $(j', j'') \in \C^2$ such that $j'$ and $j''$ are $j$-invariants of elliptic curves with complex multiplication.
\end{theor}

Here, an elliptic curve $E$ with complex multiplication (CM) is an elliptic curve that has an endomorphism ring larger than the integers. The complex affine plane is the $j$-plane, $Y_0(1) \cong \operatorname{SL}_2(\Z) \backslash \mathbb{H}$ (cf.~\cite{ben}).

The Mordell-Lang conjecture, proven by Faltings, includes the assertion that a smooth projective curve of genus $>1$ has only finitely many $K$-points over any number field $K$, as well as the Manin-Mumford conjecture which states the following \cite{zbMATH03929174}:

\begin{theor}[Raynaud]\label{mmthm}
Let $A$ be a complex abelian variety and $V\subset A$ an irreducible subvariety. $V$ is a torsion cosets $a + B$, where $B$ is an abelian subvariety of $A$ and $a \in A$ is a torsion point if an only if it contains a Zariski dense set of torsion points.
\end{theor}

Although elementary proofs for these results exist, the one based on o-minimality is the most natural for generalizations to the aforementioned three stages. A key idea, explored in \Cref{ex11} and \Cref{ex12}, links these results to functional transcendence. Specifically, the proofs leverage the following:

\begin{theor}[Ax-Lindemann, first cases]\label{althm1}
We have:
\begin{enumerate}
    \item Let $\pi = (j,j): \mathbb{H} \times \mathbb{H} \to Y_0(1) \times Y_0(1)$ be the complex uniformization map. Let $Y$ be an irreducible algebraic subvariety of $\mathbb{P}^1 \times \mathbb{P}^1$. Each irreducible component of the Zariski closure of $\pi(Y \cap \mathbb{H} \times \mathbb{H})$ is either a point, $Y_0(1)^2$, an horizontal or vertical line, $Y_0(N)$ for some $N$.
    \item Let $\pi : \C^n \to A$ be the complex uniformization map. Let $Y$ be an irreducible algebraic subvariety of $\C^n$. Then each component of the Zariski closure of $\pi(Y)$ in $A$ is a coset $a + B$, where $B$ is an abelian subvariety of $A$ and some $a \in A(\C)$.
\end{enumerate}
\end{theor}

These examples suggest why Hodge theory enters the picture: the first involves a family of elliptic curves, and the second concerns an abelian variety, corresponding to a polarized Hodge structure of type $(h^{1,0}=g,h^{0,1}=g)$. One further example brings us closer to the full conjectural framework.

First of all, an irreducible component of a Shimura subvariety of $\mathcal{A}_g$, or of its image under a Hecke operator, is called a \emph{special subvariety} of $\mathcal{A}_g$. Let $S$ be a fixed irreducible subvariety of $\mathcal{A}_g$.  An irreducible subvariety $Y \subset S$ is called \emph{special} (with respect to $S$) if it is a component of $S \cap Sh$, where $Sh$ is a special subvariety of $\mathcal{A}_g$. If $Sh_Y$ is the smallest special subvariety of $\mathcal{A}_g$ containing $Y$ (the so called ``special closure''), we call $Y$ \emph{atypical} if:
\[
\codim_{\mathcal{A}_g}(Y) < \codim_{\mathcal{A}_g}(S) + \codim_{\mathcal{A}_g}(Sh_Y).
\]

\begin{conj}[Zilber-Pink for $\mathcal{A}_g$]\label{conjzp1}
$S$ contains only finitely many maximal special atypical subvarieties.
\end{conj}

The text then continues with deeper explorations into these conjectures, their applications, and broader philosophical principles, particularly in connection to variations of Hodge structures and period maps.

\subsection{General Expectations}
Before dealing with the general case of VHS, it is worth pausing for a second to reflect on the abstract slogan given by \Cref{thmao}, \Cref{mmthm}, and \Cref{conjzp1} that we discussed so far. We are given a manifold $M$ which naturally comes with a distinguished class of submanifolds $\{ M_i \}_{i \in I}$ exhibiting a \emph{special} behaviour. Zilber-Pink-type results in such settings suggest distribution behaviour of the intersections between a fixed sub-manifold $S \subset M$ and the $M_i$s can often be understood by simple dimension criteria. In particular, one says that a component $Y$ of $S \cap M_i$ is \emph{atypical} if 
\begin{equation}\label{firstatyp}
\codim_{M}(Y) < \codim_{M}(S)+ \codim_{M}(M_i) ,
\end{equation}
and \emph{typical} otherwise. One then expects that:
\begin{itemize}
\item[(1)] $S$ contains only finitely many maximal atypical intersections.
\item[(2)] The following are equivalent: 
\begin{itemize}
\item[(i)] $S$ contains one typical intersection;
\item[(ii)] the collection of typical intersections is dense in $S$; and
\item[(iii)] there exists an $i \in I$ such that:
\begin{displaymath}
\dim S - \codim_{M}(M_i) \geq 0.
\end{displaymath}
\end{itemize}
\end{itemize}
Condition (iii) is simply asserting that a typical intersection could exists for combinatorial reasons.
\subsection{Hodge theory and period maps}\label{periodmap}
There is a version of Zilber-Pink (that recovers \Cref{conjzp1} as a special case) for arbitrary families of varieties/motives. This is given by using the formalism of variation of Hodge structures and their associated period maps, detailed in \cite[Section 4]{ben} as well as in the monograph \cite{MR2918237}. We briefly recall what we need:
\begin{itemize}
\item Let $(S,\VV)$ be a pure polarized integral VHS on a smooth quasi-projective variety $S$;
\item It corresponds to a period map 
\begin{displaymath}
 \Phi: S^{\an} \longrightarrow \Gamma \backslash D,
\end{displaymath}
in the category of complex analytic manifolds ($\Gamma \backslash D$ is in general not algebraic);
\item The period domain $D$ is is canonically an open subset (in the euclidean topology) of an algebraic variety $D^\vee$;
\item $ \Gamma \backslash D$ comes with a collection of \emph{special submanifolds}, the so called sub-Mumford-Tate domains;
\item The Hodge Locus of $(S,\VV)$ is the subset of $S$ given by the collection of preimages of the sub-Mumford-Tate domains along $\Phi$.
\end{itemize}
We will go back to this setting in \Cref{lect2}, but we can already observe that we are very well in the setting described in the previous section and that Zilber-Pink type conjectures can be at least formulated. (See \Cref{ex14} for a link with the AO problem.) Before continuing on this topic, we would like to conclude this historical overview with a couple of applications of this circle of ideas to related areas.
\subsection{Some applications}
\subsubsection{Integral points}
The first theorem is due to Lawrence and Venkatesh \cite{zbMATH07233321} and in fact was part of the motivation of Bakker and Tsimerman to prove their Ax-Schanuel conjecture (a broad generalization of \Cref{althm1} that will be discussed in details in \Cref{lect3}). As in the Preface, let $U_{n,d}$ denote the parameter space of smooth degree $d$ hypersurfaces in the $n+1$-dimensional projective space.
\begin{theor}[{\cite[Thm 10.1, Prop. 10.2]{zbMATH07233321}}]\label{lv}
There exist $n_0 \in \NN_{\geq 3}$ and a function $d_0: \NN \to \NN$ such that,
\begin{equation}\label{conditionLV}
\textnormal{ for every} \; n \geq n_0 \text{   and   } d\geq d_0(n),
\end{equation}
the set $U_{n,d}(\Z[S^{-1}])$ is not Zariski dense in 
$U_{n,d}$, for every finite set of primes $S$.
\end{theor}
See also \cite{2024arXiv240104981J} for an effective version of \eqref{conditionLV}.
\begin{rmk}
In brief (and ignoring many crucial details), the idea behind the above result is the following (closely related to the Kim-Chabauty method). If we have  a (immersive, say) period map $\Phi : S\to \Gamma \backslash D$, we we want to show that the $S$-integral points $S(\Z[S^{-1}])$ are not Zariski dense in $S$. Working $p$-adically, there's a lift of the period map $\Phi_p : S(\Q_p)\to D(\Q_p)$ and, using some results from $p$-adic Hodge theory, the authors can show that $\Phi_p(S(\Z[S^{-1}]))$ lies in an algebraic subvariety $Y $ of $D^{\vee}$. At this point, a $p$-adic version of Ax-Schanuel implies that $\Phi_p^{-1}(Y)$ can not be Zariski dense in $S$ (and in fact it has to lie in the Hodge locus). In particular $S(\Z[S^{-1}])$ cannot be Zariski dense.
\end{rmk}

Consider $f_{n,d}: X_{n,d} \to U_{n,d}$
the universal family of smooth degree $d$ hypersurfaces in
$\mathbb{P}^{n+1}$. We denote by $\VV$ the polarized $\ZZ$-variation of
Hodge structure $(R^n {f ^\an_{n, d, \CC}}_* \ZZ)_\prim$ on $U_{n, d,
  \CC}$ and by $\Phi: {U_{n,d,}}^\an \to \Gamma \backslash D$ the
associated period map. An irreducible algebraic subvariety $Y \subset
S$ is said to be {\em of positive period dimension} if $\Phi(Y_\CC^\an)$
has positive dimension. 

As the above remark outlined, \Cref{lv} is already employing tools that we will discover in our lecture. In fact, these viewpoints are even closer than one might expect. For example, in  \cite{zbMATH07745044}, using the Zilber-Pink philosophy, the authors prove:

\begin{theor}\label{maincor} 
As long as \eqref{conditionLV} is satisfied, there exists a closed strict subscheme 
$E\subset U_{n,d}$ such that, for all finite set of primes $S$, we have
\begin{displaymath}
\overline{U_{n,d}(\Z[S^{-1}])}_{\pos} \subset E,
\end{displaymath}
where $\overline{U_{n,d}(\Z[S^{-1}])}_{\pos}$ denotes the union of the irreducible
components of the Zariski closure of $U_{n,d}(\Z[S^{-1}])$ in
$U_{n,d}$ of positive period dimension.
That is: the Zariski closure of $U_{n,d}(\Z[S^{-1}]) - E(\Z[S^{-1}])$ has period dimension zero.
\end{theor}
In a nutshell, the improvement comes from understanding whether the Hodge locus (of positive period dimension) is Zariski closed or not, which brings us back to \Cref{mainquestion}.
\subsubsection{Jacobians with prescribed properties}

We recall the following folklore question.
\begin{question}
Let $g\geq 4$, and let $ A$ be a generic abelian variety over an algebraically closed field $k$. What is the lowest integer $g'$ such that there exists a Jacobian $J$ of dimension $g'$ and a surjection $J\to A$?
\end{question}

In characteristic zero, the first results were obtained by Chai-Oort, Tsimerman, and after in a more general setting by Masser-Zannier \cite{zbMATH07168647}. For simplicity we just cite the most recent and general one, due to Tsimerman \cite{2023arXiv230205860T} (see also references therein):
\begin{theor}
For any two integers $g\geq 4$ and $g' \leq 2g -1$, there exist $g$-dimensional abelian varieties over $\overline{\Q}$ which are not quotients of a Jacobian of dimension $g'$.
\end{theor}
All known proofs use o-minimality and are inspired once more by the Pila-Zannier strategy (that we discussed in the exercises).

 On the other hand, Mumford \cite{zbMATH03271259} shows that there exist
principally polarized abelian 4folds $A$ with trivial endomorphism ring, that are not Hodge
generic in $\cA_4$ (they have an exceptional Hodge class in
$H^4(A^2, \ZZ)$). A question often attributed to Serre is to describe ``as
explicitly as possible'' such abelian varieties \emph{of Mumford's
  type}. The most satisfying way would be to show the existence of a smooth projective curve over $\overline{\QQ}$ of genus $4$,
whose Jacobian is of Mumford's type.

The Zilber-Pink philosophy is the crucial ingredient behind the following \cite[Thm. 3.17]{2021arXiv210708838B}, as we will explore in \Cref{ex399}:
\begin{theor}\label{serrequestion}
 There exists a smooth projective curve $C/\Qbar$ of genus $4$ whose
 Jacobian is of Mumford's type, i.e. it has Mumford-Tate group isogenous to a $\Q$-form of the complex group $\Gm \times \Sl_2\times \Sl_2 \times \Sl_2$. In fact such curves are dense in the moduli space $\mathcal{M}_4$.
\end{theor}

For more about the Hodge locus of $\mathcal{M}_g$ and the Zilber-Pink conjecture in this special case, we refer to the forthcoming survey \cite{specialcurves}.

\newpage
\section{Lecture 2: Distribution of the Hodge locus}\label{lect2}
After the historical overview and the motivations discussed from number theoretic problems, we now focus solely on the complex analytic side.

We begin with some general notation and conventions. Here an algebraic variety $S$ is a reduced scheme of finite type over the field of complex numbers, but may be reducible. If $S$ is an algebraic (resp.\ analytic) variety, by a subvariety $Y \subset S$ we always mean a \emph{closed} algebraic (resp.\ analytic) subvariety. The smooth locus of $Y$ is denoted by $Y^{\operatorname{sm}}$.

A $\Q$-Hodge structure of weight $n$ on a finite dimensional $\Q$-vector space $V_\Q$ is a decreasing filtration $F^\bullet$ on the complexification $V_\C$ such that
\[
V_\C= \bigoplus_{p\in \Z} F^{p} \oplus \overline{F^{n-p}}.
\]
The category of pure $\mathbb{Q}$-Hodge-structures is Tannakian (and semisimple if we consider polarizable ones, as we will usually do). The Mumford--Tate group $\MT(V) \subset \GL(V)$ of a $\mathbb{Q}$-Hodge structure $V$ is the Tannakian group of the Tannakian subcategory $\langle V\rangle ^\otimes$ of $\mathbb{Q}$-Hodge structures generated by $V$. Equivalently, $\MT(V)$ is the smallest $\mathbb{Q}$-algebraic subgroup of $\GL(V)$ whose base-change to $\mathbb{R}$ contains the image of $h: \mathbb{S} \to \GL(V_{\mathbb{R}})$. It is also the stabilizer in $\GL(V)$ of the Hodge tensors for $V$. As $V$ is polarised, this is a reductive algebraic group. Cf. \Cref{ex21}.
 
\subsection{Typical and atypical intersections}

Recall that, from \Cref{periodmap} and \cite{ben}, to understand a VHS $(S,\VV )$, we consider the associated holomorphic period map 
\begin{equation} \label{period0}
 \Phi: S^{\an} \longrightarrow \Gamma \backslash D, \ \ s \mapsto [\VV_s]
\end{equation}
which completely describes $\VV$.
Here $(\G, D)$ denotes the generic Hodge datum of $\VV$ and $\Gamma \backslash D$ is the associated Hodge variety. The
Mumford--Tate domain\footnote{It is worth elucidating an important and perhaps confusing difference in the literature between \emph{period domains} and Mumford-Tate domains. In \cite{ben}, the period domains are defined as orbits under the group of automorphism of a polarized Hodge structure $(H_\Q,q_\Q)$, where $q_\Q$ is a $(-1)^{\operatorname{weight}(H)}$ symmetric bilinear form. However, a VHS can actually have extra fixed Hodge tensors. The simplest example is a family of squares of elliptic curves. The term MT domain, refers to the fact that we always normalise this issue by taking as ambient group $\mathbf{G}$ the generic Mumford-Tate group of the VHS (which is nothing but the MT of $\VV_s$ for a very general point $s\in S$). In particular $\mathbf{G}(\R)$ can be more general than an orthogonal or a symplectic group. This point will implicitly come back when dealing with \emph{weakly special subvarieties}, cf. \Cref{lect3}} $D$ decomposes as a product $D_1 \times \cdots
\times D_k$, according to the decomposition of the adjoint group $\G^\ad$ into a product $\G_1
\times \cdots \times \G_k$ of simple factors (notice that some factors
$\G_i$ may be $\RR$-anisotropic). Replacing $S$ by a finite \'etale
cover and reordering the factors if necessary, the lattice $\Gamma \subset
\G^\ad(\R)^+$ decomposes as a direct product $\Gamma
\cong \Gamma_1 \times \cdots \times \Gamma_r$, $ r \leq k$, where $\Gamma_i\subset
\G_i(\R)^+$ is an arithmetic lattice for each $i$, and thus is Zariski-dense in
$\G_i$. Writing $D^\prime = D_{r+1} \times \cdots \times D_k$ for the
product of factors where the monodromy is trivial (in particular, $D^\prime$ contains in
particular all the factors $D_i$ for which $\G_i$ is
$\RR$-anisotropic), the period map is written
\begin{equation} \label{period}
 \Phi: S^{\an} \longrightarrow \Gamma \backslash D \cong \Gamma_1 \backslash D_1
 \times \cdots \times \Gamma_r
 \backslash D_r \times D'\;\;,
\end{equation}
and the projection of $\Phi(S^\an)$ on $D^\prime$ has image a point. See also the structure theorem for VHS from \cite[(III.A.2)]{MR2918237}.

We distinguish between special subvarieties $Z$
of \emph{zero period dimension} (i.e., $\Phi(Z^\an)$ is a point of $\Gamma
\backslash D$), which are geometrically elusive; and those of \emph{positive
period dimension} (i.e., $\dim_\CC \Phi(Z^\an) >0$).

\begin{defi} \label{positive period dimension} \label{fpositive} \hfill
 \begin{enumerate}
  \item
 A subvariety $Z$ of $S$ is said to be of {\em positive period dimension for $\VV$} if $\Phi(Z^{\an})$ has
 positive dimension.
 
 \item
 The {\em Hodge locus of positive period dimension} 
 $\HL(S, \VV^\otimes)_{\pos}$ is the
 union of the special subvarieties of $S$
 for which $\VV$ has positive period dimension.
\end{enumerate}
\end{defi}

The Cattani-Deligne-Kaplan theorem \cite{CDK} recalled in the introduction and discussed in \cite{ben} equips the Hodge locus with a structure of a countable (possibly finite) union of subvarieties of $S$.

Using period maps, special subvarieties 
can also be defined as {\em intersection loci}. Indeed, a closed irreducible subvariety $Z \subset S$ is special
for $\VV$ (we will equivalently say that it is \emph{special} for $\Phi$)
precisely when $Z^{\an}$ coincides with an analytic irreducible 
component $\Phi^{-1}(\Gamma' \backslash D')^0$ of $\Phi^{-1}(\Gamma' \backslash D')$, for $(\G', D') \subset
(\G, D)$ the generic Hodge sub-datum of $Z$ and $\Gamma' \backslash D'
\subset \Gamma \backslash D$ the associated Mumford-Tate subdomain.

\begin{defi}\label{atypical}
 Let $Z = \Phi^{-1}(\Gamma' \backslash D')^0\subset S$ be a special
 subvariety for $\VV$ with generic Hodge datum $(\G', D')$. Then $Z$ is
 said to be \emph{atypical} if $\Phi(S^{\an})$ and $\Gamma'
 \backslash D'$ do not intersect generically along $\Phi(Z)$. That is, $Z$ is atypical when
 \begin{equation} \label{equation atypical}
  \codim_{\Gamma\backslash D} \Phi(Z^{\an}) < \codim_{\Gamma\backslash
  D} \Phi(S^{\an}) + \codim_{\Gamma\backslash D} \Gamma'\backslash
 D'\;\;.
 \end{equation}
Otherwise $Z$ is said to be \emph{typical}. 
\end{defi}
The \emph{atypical Hodge locus} $\HL(S,\VV^\otimes)_{\atyp} \subset \HL(S, \VV^\otimes)$
 (resp.\ the \emph{typical Hodge locus} $\HL(S,\VV^\otimes)_{\typ} \subset \HL(S, \VV^\otimes)$) is
 the union of the atypical (resp.\ strict typical) special subvarieties of $S$ for $\VV$. The Hodge locus is the recovered by the union of its typical and atypical parts.

\subsection{Conjectures}
Let $\VV$ be a polarizable $\ZZ$VHS on an irreducible smooth quasi-projective variety $S$, from \cite{2021arXiv210708838B} (which refines \cite{klin}, as well as the discussion from \Cref{lect1}) we expect:

\begin{conj}[Zilber--Pink conjecture for the atypical Hodge locus, strong version] \label{main conj}
 The atypical Hodge locus
 $\HL(S,\VV^\otimes)_{\atyp}$ is a finite union of atypical special
subvarieties of $S$ for $\VV$.
\end{conj}
\begin{rmk}\label{rmklevel3}
The above conjecture is a vast generalization of the case of family of abelian varieties that we discussed before (\Cref{conjzp1}), as well as the notable study of the Noether-Lefschetz locus (cf. \Cref{ex27}). It is worth mentioning that these two examples are the cases of \emph{level 1 and 2} (a refinement of the weight of the Hodge structure, cf \cite[Sec. 4.6]{2021arXiv210708838B}). In the literature there is also a case of study in level 3, but very little (nothing, to the best of my knowledge) was speculated in level >3.

The level 3 case, is intimately related to String theory. Indeed their interest in Calabi-Yau 3-manifolds stems from their connection to conformal field theories (CFTs). Gukov and Vafa \cite{zbMATH02054562} posed a question regarding the existence of infinitely many Calabi-Yau manifolds with complex multiplication of a fixed dimension, motivated by the connection to rational conformal field theories (RCFT), introduced earlier by Friedan-Qiu-Shenker. Around the same time, Moore explored the arithmetic-string theory connection, particularly the role of attractor varieties in black hole constructions within IIB string theory. These investigations revealed connections to atypical intersections that are not necessarily of CM type. For more recent results, see also \cite{zbMATH07323024}.
\end{rmk}
\begin{conj}[Density of the typical Hodge locus] \label{conj-typical}
 If 
$\HL(S, \VV^\otimes)_\typ$ is not empty then it is dense (for the analytic topology) in
$S$.
\end{conj}
At first sight, these conjectures might look formal and might be hard to grasp their meaning. We refer to \Cref{sectionex2} for a collection of exercises on special and concrete cases.
\subsection{Results--atypical intersections}
 As above, we let $(S,\VV)$ be a polarizable $\ZZ$VHS on a smooth connected complex quasi-projective variety
$S$, and we denote by $(\G, D)$ the generic Hodge datum. In \Cref{lect3} describe a proof of the following \cite[Thm. 3.1]{2021arXiv210708838B}:

\begin{theor}[Geometric Zilber--Pink]\label{geometricZP}
Let $Z$ be an irreducible component of the Zariski closure
of the union of the atypical special subvarieties of positive period dimension in $S$. Then either
\begin{itemize}
 \item[(a)] $Z$ is a maximal atypical special subvariety, or
  \item[(b)] the adjoint Mumford--Tate group $\G_Z^\ad$
 decomposes as a nontrivial product $\HH^\ad_Z \times \LL_Z$, $Z$ contains a Zariski-dense
set of fibers of $\Phi_{\LL_{Z}}$ which are atypical weakly special
subvarieties of $S$ for $\Phi$, where (possibly up to an \'{e}tale covering)
\[
\Phi_{|Z^\an}= (\Phi_{\HH_{Z}}, \Phi_{\LL_{Z}}): Z^\an \longrightarrow \Gamma_{\G_{Z}}\backslash D_{G_{Z}}= \Gamma_{\HH_{Z}}
\backslash D_{H_{Z}} \times \Gamma_{\LL_{Z}}\backslash D_{L_{Z}} \subset \Gamma \backslash
D,
\]
and $Z$ is Hodge generic in a special subvariety $\Phi^{-1}(
\Gamma_{\G_{Z}}\backslash D_{G_{Z}})^0$ of $S$ for $\Phi$ which is monodromically typical (and therefore typical). 
\end{itemize}
\end{theor}

\subsection{Results--typical intersections}
We first describe two results that don't use functional transcendence (whose proofs can be found in  \cite[Sec. 7 and 10]{2021arXiv210708838B}).
\begin{theor}\label{typicallocus}
If the typical Hodge locus $\HL(S,\VV^\otimes)_{\typ,\pos} $ is nonempty then
$\HL(S,\VV^\otimes)_{\typ,\pos}$ is analytically (hence Zariski) dense in
$S$.
\end{theor}

\begin{theor} \label{level criterion}
Let $\VV$ be a polarizable $\ZZ$VHS
on a smooth connected complex quasi-projective variety 
$S$, with generic Hodge datum $(\G, D)$ and algebraic monodromy group
$\HH$. Suppose that $\HH = \G^\der$. If $\VV$ is of level at
 least $3$ then $\HL(S,\VV^\otimes)_{\typ} = \emptyset$ (and thus $\HL(S,
 \VV^\otimes)= \HL(S, \VV^\otimes)_\atyp$).
\end{theor}

\begin{rmk}
The above theorem roughly states that outside the case of surfaces and abelian varieties, every special subvariety is an atypical intersection. This is of course an oversimplification of the concept of level. The proof of the theorem rests upon the following representation theoretical result which may be helpful to state, see  \cite[Prop. 7.5]{2021arXiv210708838B}). Suppose $\fg_\RR$ is a simple $\RR$-Hodge-Lie algebra generated in level
  $1$ and of level at least $3$ (i.e.. $\fg^2=[\fg^1,\fg^1]$, $\fg^3=[\fg^{1},\fg^2]$ and so on, and $\fg^3\neq 0$). If $\fg'_\RR \subset \fg_\RR$ is an $\RR$-Hodge-Lie subalgebra
  satisfying ${\fg'}^i = \fg^i$ for all $|i|\geq 2$ then $\fg'=
  \fg$.

\end{rmk}

We record here two results that use once more some functional transcendence. Hodge theory actually gives a simple combinatorial criterion to decide whether $\HL(S,\VV^\otimes)_{\typ}$ is empty or not. Indeed see the recent works of Eterovic-Scanlon, Khelifa-Urbanik \cite{2022arXiv221110592E, 2023arXiv230316179K}. They defined:

\begin{defi}
A strict Hodge sub-datum $(\mathbf{H}, D_H) \subset (\mathbf{G}, D_G)$ is said $\VV$-\emph{admissible} if
\begin{displaymath}
\dim \Phi (S^{\operatorname{an}})  + \dim D_M \geq \dim D.
\end{displaymath}
\end{defi}

\begin{theor}[Eterovic-Scanlon, Khelifa-Urbanik]\label{dense}
If $(\mathbf{H}, D_H) \subset (\mathbf{G}, D_G)$ is a $\VV$-admissible Hodge sub-datum, then 
\begin{displaymath}
\HL(S, \VV^{\otimes}, \mathbf{H}):=\{s\in S : \exists g \in \mathbf{G}(\Q)^+, \MT(\VV_s)\subset g \mathbf{H} g^{-1}\}
\end{displaymath}
 is dense in $S^{\operatorname{an}}$.
\end{theor}
(We will come back to this interesting statement in \Cref{ex37}.)

On a related direction, there is another result (taken from \cite{2024arXiv240203601B}), which builds on the geometric Zilber-Pink to characterise VHS from their Hodge Locus (when possible).
\begin{defi}\label{def:isog}
Let $\mathbb{V}_1, \VV_2\in \Z \text{VHS}/S$. We say that $\VV_1$ and $\VV_2$ are \emph{isogenous} if there is an equivalence of tensor categories $\langle \VV_{1,\Q} \rangle^\otimes \cong \langle \VV_{2,\Q} \rangle^\otimes$, where $ \langle \VV_{i,\Q} \rangle^\otimes$ denotes the smallest Tannakian subcategory of $\Q$VHS's containing $\VV_{i,\Q}$. 
\end{defi}
Of course the Tannakian categories $\langle \VV_{i,\Q} \rangle^\otimes$ appearing above are equivalent (as tensor categories) to the category of finite dimensional representations of their generic Mumford-Tate group. (The equivalence is realised by the functor of $\otimes$-automorphisms of the fiber functor). It can happen that two VHS have isomorphic Mumford-Tate groups, but the isomorphism does not induce an equivalence of tenor categories. cf. Def. 1.10 and Thm. 2.11 in the article of Deligne and Milne \cite{zbMATH03728195}.
\begin{rmk}\label{rmkonabelian}
If two complex principally polarized abelian varieties $A,B$ that are isogenous in the usual sense (either via a polarized or an unpolarized isogeny), then the Hodge structures $H^1(A,\Z), H^1(B,\Z)$ are also isogenous in the sense of \Cref{def:isog} (here we take as base $S$ the spectrum of $\C$). However the converse is in general not true, for example $H^1(A,\Z)$ is isogenous to $H^1(A\times A, \Z)$. However, for two principally polarized $g$-dimensional abelian varieties with Mumford--Tate group $\mathbf{GSp}_{2g}$ the two notions agree, since this is the case when the Mumford--Tate group is as big as possible.
\end{rmk}

\begin{theor}[{\cite[Thm. 1.18]{2024arXiv240203601B}}]\label{hodgethm}
Let $S$ be a smooth quasi-projective variety and $\mathbb{V}_1, \mathbb{V}_2$ two pure polarized $\Z$VHSs on $S$. Assume that the generic Mumford--Tate groups of $\VV_1$ and $\VV_2$ are $\QQ$-simple. 
If 
\[
\HL(S,\mathbb{V}_1^\otimes)_{\pos, \typ}=\HL(S,\mathbb{V}_2^\otimes)_{\pos, \typ}\neq \emptyset,
\]
 then $\mathbb{V}_1$ is isogenous to $\mathbb{V}_2$. As a consequence, $\HL(S,\mathbb{V}_1^\otimes)=\HL(S,\mathbb{V}_2^\otimes)$.
\end{theor}

\newpage

\section{Lecture 3: Ax-Schanuel and applications}\label{lect3}
We are finally ready to discuss the topic of \emph{functional transcendence} and the main theorem of the theory: Ax-Schanuel. It will provide a vast generalization of \Cref{althm1}.
\subsection{Statements}
Let $S$ be a smooth complex quasi-projective variety supporting a
polarizable $\Z$VHS $\mathbb{V}$, with generic Hodge datum $(\G, D)$, 
algebraic monodromy group $\HH$, and period map $\Phi: S \to \Gamma_{\HH}
\backslash D_{H} \subset \Gamma \backslash D$. Without loss of
generality we can assume, replacing if necessary $S$ by a finite
\'etale cover, that $\Gamma$ is torsion-free. Let $\tilde{\Phi}:
\widetilde{S^\an} \to D_H$ be the lift of $\Phi$ at the level of universal coverings.

As already discussed, the domain $D_H$ is canonically embedded as an
open complex analytic real semi-algebraic subset in a flag variety
$D_H^{\vee}$ (called its \emph{compact dual}). For example, let $\mathcal{L}=\mathcal{L}(V_\Q , \varphi)$ denote the symplectic Grassmanian of Lagrangian,  i.e. the scheme parametrizing maximal isotropic subspaces of $V_\Q$ with respect to the form $\varphi$. The map $\mathbb{H}_g\to \mathcal{L}(\C)$ that sends a Hodge structure $y\in \mathbb{H}_g$ to the corresponding Hodge filtration $\operatorname{Fil}^0 \subset V_\C$ is the open immersion (the Borel embedding) of the $g$-dimensional Siegel space into its compact dual.

 We define
an irreducible algebraic subvariety of $D_H$ (resp. $S \times D_H$) as a complex analytic
irreducible component of the intersection of an algebraic subvariety
of $D_H^{\vee}$ (resp. $S \times D_{H}^{\vee}$) with $D_H$ (resp. $S
\times D_{H}$).

We have already introduce the special subvarieties of $(S,\VV)$ in this context but in the sequel, and in general when ignoring the arithmetic side of the story, there is a better suited concept of \emph{weakly special} subvarieties. 
\begin{defi}
The weakly special subvarieties of $(S,\VV)$ are the closed irreducible algebraic subvarieties $Y\subset S$ maximal among the closed irreducible algebraic subvarieties $Z$ of $S$ whose algebraic monodromy group $\mathbf{H}_Z$ with respect to $\VV$ equals $\mathbf{H}_Y$. 
\end{defi}
Any special subvariety (defined in \Cref{lect2}) is weakly special but the converse does not hold. The reader might wish to consult \cite[Sec. 4.4]{2021arXiv210708838B} for more details, and properties, as well as \Cref{ex31} and \Cref{ex32} (which gives the basis to understand the link between the different notions).

The following result is the so called \emph{Ax-Schanuel Theorem} (for pure variations of Hodge structures). It was conjectured by B. Klingler \cite[Conj.
7.5]{klin} and later proved by Bakker and Tsimerman \cite[Thm.
1.1]{MR3958791}, generalising the work of Mok, Pila and Tsimerman
\cite[Thm. 1.1]{as} from level one to arbitrary levels.

\begin{theor}[Ax-Schanuel for VHS]\label{astheorem}  
Let $W \subset S \times D_H$ be an algebraic subvariety. Let $U$ be an
irreducible complex analytic component of $W \cap S \times_{\Gamma_\HH \backslash D_H} D_H$
such that 
\begin{displaymath}
\codim_{S \times D_{H}} U< \codim_{S \times D_{H}} W+ \codim_{S
  \times D_{H}} (S \times_{ \Gamma_{\HH} \backslash D_{H}}
D_H)\;\;.  
\end{displaymath}
Then the projection of $U$ to $S$ is contained in a strict weakly
special subvariety of $S$ for $\VV$.
\end{theor}
Notice that $S \times_{ \Gamma_{\HH} \backslash D_{H}} D_H$ is simply the image of the graph of $\tilde{\Phi}:
\widetilde{S^\an} \to D_H$ under $
\widetilde{S^\an} \times D_H\to S \times D_H$.
\begin{rmk}
The intersection $W
  \cap S \times_{\Gamma_\HH \backslash D_H} D_H$ can be identified
  with the intersection in $S \times D_H$ between $W$ and the the
  image of $ \widetilde{S^{\an}}$ in $S \times D_H$ along the map
  $(\pi, \tilde{\Phi})$.
  \end{rmk}
  A special case might be already interesting to observe:
  \begin{cor}[Special case of AS in the period domain, Bakker-Tsimerman]\label{specialAS}
Let $Y\subset D^{\vee}$ be an algebraic subvariety of codimension $\geq \dim S$ and denote by $Y_\Gamma$ the image of $Y \cap D$ in $\Gamma \backslash D$. Each component of $\Phi^{-1}(Y_\Gamma)$ lies in a strict weakly special subvariety of $S$.
\end{cor}
In the conclusion, in particular, we observe that $\Phi^{-1}(Y_\Gamma)$ is not Zariski dense in $S$.

\subsection{Proof of the geometric Zilber-Pink conjecture}
All finiteness results described in the lectures come from the following: the only countable and definable sets (in some o-minimal structure over $\R$), are finite sets. Behind the proof \Cref{geometricZP} there are two basic observations:
\begin{enumerate}
\item Ax-Schanuel \Cref{astheorem} gives the countability;
\item The theory of degenerations of VHS gives the definability. 
\end{enumerate}
\begin{rmk}
Different approaches can be found in \cite{binyamini2021effective, Daw19, 2024arXiv240616628B}. They still use (1) from above, but not (2). Nevertheless the key idea is similar and boils down to the fact that a constructible set which is a countable union of algebraic subvarieties can be expressed as a finite union.
\end{rmk}
\subsubsection{Definable fundamental sets, in Hodge theory}
Let $\Phi: S  \to \Gamma_\HH \backslash D_H$ be the period map for $\VV$. As recalled,
the Mumford-Tate domain $D_H=H(\R)^+/M$ is a real semi-algebraic open subset of
its compact dual $D_H^\vee$. Recall also that $S$  it is naturally definable in any extension of
the o-minimal structure $\R_{\operatorname{alg}}$, since it is an algebraic variety. {\it From now on,
definable will be always understood in the o-minimal structure
$\R_{\an, \exp}$.}

Let us introduce a notion of ``definable fundamental set''  of $S$ for $\Phi$, arguing as at the beginning of
\cite[Section 3]{MR3958791}. Let $(\overline{S},E)$ be a log-smooth
compactification of $S$, and choose a definable atlas of $S$ by
finitely many polydisks $\Delta^k \times (\Delta ^*)^\ell$. Let 
\begin{displaymath}
\exp : \Delta^k \times \Hh^\ell \to \Delta^k \times (\Delta ^*)^\ell
\end{displaymath}
be the standard universal cover, and choose 
\begin{displaymath}
\Sigma = [-b,b] \times
[1, + \infty[
\subset \Hh
\end{displaymath}
such that $\Delta^k \times \Sigma ^\ell$ is a fundamental
set for the $\ZZ^{\ell}$-action by covering transformation. Let $\F$ be the disjoint union of
$\Delta^k \times \Sigma^\ell $ over all charts and choose lifts
$\Delta^k \times \Hh^\ell  \to D$ of the period map restricted to each
chart to obtain a lift $\tilde{\Phi}=\tilde{\Phi}_\F : \F
\to D_H$. From the Nilpotent Orbit Theorem \cite[(4.12)]{sch73} we
obtain the following diagram in the category of definable complex
manifolds: 

\begin{center}
\begin{tikzpicture}[scale=2]
\node (A) at (-1,1) {$\F$};
\node (B) at (0,1) {$D_H$};
\node (C) at (-1,0) {$S^{\an}$};
\path[->,font=\scriptsize,>=angle 90]
(A) edge node[above]{$\tilde{\Phi}_\F$} (B)
(A) edge node[right]{$\exp$} (C);
\end{tikzpicture}.
\end{center}

\subsection{Main steps of the proof}
Let $\Omega$ be a fixed finite set of representatives for the $H(\R)^+$-conjugacy classes of semisimple algebraic subgroups of $H_\R$ that are without compact factors.

Set:
\begin{displaymath}
\Pi_0:=\{(x,g,M) \in \F \times H(\R)^+\times \Omega : \tilde{\Phi}(x) \in
g\cdot D_M \}
\end{displaymath}
(the last condition means that the image of the Deligne torus via $\tilde{\Phi}(x)$ lies in $gG_M g^{-1}$). We remark that all Hodge data we have to consider, belong to the above definable set.

 Notice that if $(x, g,F) \in \Pi_0$ then $g \cdot M
g^{-1}\tilde{\Phi}(x) = g \cdot D_M$ is an algebraic subvariety of $D$. 

By ``dimension'' we will always mean the complex
dimension (and possibly the local dimension at some point $y$, which we denote by $\dim _y$). Consider the functions:

\begin{displaymath}
d : \Pi_0 \to \RR, \qquad (x,g,M) \mapsto d(x, g,F):=
\dim_{\tilde{\Phi}(x)} (g \cdot D_M ). 
\end{displaymath}

\begin{displaymath}
d_{S} : \Pi_0 \to \RR, \qquad (x,g,M) \mapsto d(x, g,F):=
\dim_{\tilde{\Phi}(x)} \left(g \cdot D_M \cap \tilde{\Phi}(\F)\right). 
\end{displaymath}

To ease the notation, from now on, by $x$ we mean $\tilde{\Phi}(x)$ (and we don't distinguish between groups and their real points).
\begin{defi}\label{defsets}
We define two sets:
\begin{displaymath}
\Pi_1:=\{(x,g,M) \in \Pi_0 : \forall (x,g_1,M_1)\in \Pi_0 : gMg^{-1}x \varsubsetneq g_1 M_1 g_1^{-1}x \Rightarrow
\end{displaymath} 
\begin{displaymath}
d(x,g,M)-d_S(x,g,M)<d(x,g_1,M_1)-d_S(x,g_1,M_1) \}
\end{displaymath}
and
\begin{displaymath}
\Pi_2:=\{(x,g,M) \in \Pi_1 : \forall (x,g_2,M_2)\in \Pi_0 : g_2 M_2g_2^{-1}x \varsubsetneq  gMg^{-1}x \Rightarrow
\end{displaymath} 
\begin{displaymath}
d_S(x_2,g_2,M_2) < d_S(x,g,M)\}
\end{displaymath}
\end{defi}

The first step towards \Cref{geometricZP}, is the following
\begin{theor}\label{thm00}
The set $\{gMg^{-1}: (x,g,M)\in \Pi_2\}$ is finite.
\end{theor}

In the direction of the above theorem, we simply prove the following, using several times Ax-Schanuel (namely \Cref{astheorem}). The rest of the deduction will be spread out in the exercises (cf. 5.3.6-5.3.8).

\begin{prop}\label{pro110}
If $(x,g,M)\in \Pi_2$, then $gMg^{-1} \cdot x$ is a weakly special subvariety of $D$.
\end{prop}

\begin{proof}
Let $A$ be an irreducible analytic component of $\tilde{S}\cap gMg^{-1}x$ at $x$, such that $\dim A= d_S(x,g,M)$. By \Cref{astheorem} the $\pi (A)$ lies in some strict weakly special subvariety. Let $S'\subset S$ be the smallest with this property, and $(H',D')$ its monodromy datum. So $A \subset D'$, and, by construction, $A \subset gMg^{-1}x$. Consider the intersection between Zariski closed subvarieties in the compact dual $D^\vee$
\begin{displaymath}
D'\cap gMg^{-1}x
\end{displaymath}
and let $Y$ be a component containing $A$. By construction $Y$ is associated to a triple $(x,g_Y,M_Y)\in \Pi_0$ (and clearly $g_Y M_Yg_Y^{-1} \cdot x \subset gMg^{-1} \cdot x$)\footnote{Here there's a hidden claim: the components of intersections of two orbits under semisimple subgroups are again orbits of some subgroup.}. We observe that
\begin{equation}
d_S(x,g,M)=\dim A = d_S(x,g_Y,M_Y)
\end{equation}
By the fact that $(x,g,M)$ belongs to $\Pi_2$, we must have:
\begin{displaymath}
g_Y M_Yg_Y^{-1}\cdot x= g M g^{-1}\cdot x
\end{displaymath}
Since, by construction, $g_Y M_Yg_Y^{-1}\cdot x \subset D'$, we have learnt (essentially from AS) that $g M g^{-1}\cdot x\subset D'$.

We claim that we also have
\begin{equation}\label{eq00}
\dim D' - \dim D' \cap \tilde{S}  \leq d(x,g,M)-\dim A= d(x,g,M) - d_S(x,g,M).
\end{equation}

Heading for a contradiction, suppose $\dim D' - \dim D' \cap \tilde{S}  > d(x,g,M)-\dim A= d(x,g,M) - d_S(x,g,M).$ Observe that $\dim D' \cap \tilde{S}=\dim S'$. The $>$ implies that $S'\times g M g^{-1}\cdot x \subset S' \times D'$ is an atypical intersection, but this would contradict the minimality of $S'$.

Since $\dim D' - \dim D' \cap \tilde{S} = d(x,g_1,M_1)-d_S(x,g_1,M_1)$ (for some $(x,g_1,M_1)\in \Pi_0$ associated to $D'$), the fact that  $(x,g,M)\in \Pi_1$ implies that 
\begin{displaymath}
g Mg^{-1}x = D'(=g_1M_1g_1^{-1}x)
\end{displaymath}
as desired.
\end{proof}

\begin{rmk}
The proof presented above is closer in spirit to \cite{MR3867286} rather than \cite{2021arXiv210708838B} since it deals with all possible sources of special subvarieties at once, rather than giving a complicated induction/algorithm to run trough the proof. This is also closer to what happens in the most general case discussed in \cite{2024arXiv240616628B} (where several of the above ideas are organised around a statement of Ax-Schanuel type \emph{in families}). The very first appearance of these ideas in the setting of VHS, beyond Shimura varieties, can be found in \cite[Sec. 6]{BU}.
\end{rmk}

\subsection{Stronger versions of Ax-Schanuel}
We conclude by presenting two recent results that imply \Cref{astheorem}.

We start with the one in the so called \emph{period torsor}. We fix a base point $s_0\in S$. Let $\VV_0$ be the trivial variation whose fiber is the fiber of $\VV$ over $s_0$. Consider the variation $\mathbb{E} := \Hom(\VV,\VV_0)$, its underlying algebraic flat vector bundle with total space $\mathcal{E}$, and let $\mathbb{I}\subset \mathcal{E}$ be the open set of isomorphisms of the fibers (as vector spaces) in the geometric total space, which is naturally a $\GL(\VV_{\C,0})$-torsor over $S$ by post-composition. 

Let $S'$ be the minimal covering space of $S(\VV)$ which trivializes the local system associated to the VHS. Solving the connection naturally gives a flat section $\sigma_{\VV} : S' \to \mathbb{I}$ by sending a path to its flat transport operator. We denote the image we denote by $\Sigma_\VV$ which we think of as a flat leaf. It turns out that its Zariski closure $\Omega_\VV:=\overline{\Sigma_\VV}^{Zar} $ is the full monodormy orbit of $\Sigma_\VV$, and is therefore naturally a $G=\mathbf{G}(\C)$-torsor which we call the period torsor. (Here $ \mathbf{G}$ is the monodromy of $(S,\VV)$, previously denoted by $\mathbf{H}$).
\begin{theor}[{\cite[Thm. 1.3]{2022arXiv220805182B}}]\label{asperiod}
Suppose $W\subset \Omega_{\VV}$ is an algebraic subvariety and $U$ a component of
$W\cap \Sigma_\VV$ such that
\begin{displaymath}
\codim_W U < \dim G.
\end{displaymath}
Then the projection of $U$ to $S$ is contained in a strict weakly special subvariety of $S$.
\end{theor}
Bakker and Tsimerman offer two proofs of the above theorem: one generalizing the proof of \Cref{astheorem} and using o-minimality and another one using the work \cite{2021arXiv210203384B} in differential algebraic geometry. We conclude just by briefly mentioning the approach of Bl\'{a}zquez-Sanz, Casale, Freitag, and Nagloo, which implies all AS statements discussed so far.
\begin{theor}[Foliated Ax-Schanuel, {\cite{2021arXiv210203384B}}]\label{thm:newAS}
Let $G$ be a semisimple complex group, and let $\nabla$ be a $G$-principal connection on $\pi : P \to S$ with Galois group $G$. Let $V$ be a subvariety of $P$, $x\in V$, and let $\mathcal{L} \subset \hat{P}_x$ be a formal horizontal leaf through $x$. Let $W$ be an irreducible component of $V \cap \mathcal{L}$. If
\begin{displaymath}
\dim V < \dim W + \dim G,
\end{displaymath}
then the projection of $W$ in $S$ is contained in a finite union of $\nabla$-special subvarieties.
\end{theor}
(In fact their statement works over any field of characteristic zero, and $G$ is simply required to be \emph{sparse}). In such a generality, the study of $\nabla$-special subvarieties is not easy, and we refer to the work of Bl\'{a}zquez-Sanz, Casale, Freitag, and Nagloo for more details.

\newpage

\section{Lecture 4: Algebraicity and quasiprojectivity of images of period maps}\label{lect4}
This final lecture is a follow up of the lecture series of B. Bakker \cite{ben}. We discuss the main results from \cite{gaga}.
Definable GAGA was already presented in the lectures of C. Miller at CIRM, but nevertheless, we recall its statement for future reference (see {\cite[Thm. 1.4]{gaga}} as well as \cite[Sec. 3]{ben}).

\begin{theor}[O-minimal GAGA]
Let $S$ be a separated algebraic space of finite type over and $S^\deff$ the associated definable complex analytic space. The \emph{definabilization} functor 
\begin{displaymath}
\operatorname{Coh}(X)\to \operatorname{Coh}(X^{\deff})
\end{displaymath}
 is fully faithful, exact, and its essential image is closed under subobjects and quotients.
\end{theor}
The goal of the final lecture is to explain applications to VHS of such results. A notable example:
\begin{theor}[{\cite[Thm. 1.1]{gaga}}]
Let $S$ be a reduced separated algebraic space of finite type $\Phi: S^{\an}\to \Gamma \backslash D$ a period map. Then
\begin{itemize}
\item $\Phi$ factors (uniquely up to unique isomorphism) as $\Phi=\iota \circ f^{\an}$ where $f: S \to Y$ is a dominant map of (reduced) finite-type algebraic spaces and $\iota : Y^{\an}\to \Gamma \backslash D$ is a closed immersion of analytic spaces;
\item the Griffiths $\Q$-line bundle $L:=\otimes_i \det F^i$ restricted to $Y$ is the analytification of an ample algebraic $\Q$-bundle, and in particular $Y$ is a quasi-projective variety.
\end{itemize}
\end{theor}
An interesting application is the following:
\begin{cor}
Let $\mathcal{M}$ be a reduced separated Deligne-Mumford stack of finite type admitting a quasi-finite period map. Then its coarse moduli space is quasi-projective.
\end{cor}

A key input in the proof is a very general result on \emph{definable images}:
\begin{theor}
Let $S$ be an algebraic space, $\mathcal{Y}$ a definable analytic space, and $f: S^{def}\to \mathcal{Y}$ an analytically proper definable analytic map. Then there exists a unique factorization of $f$ 
\begin{displaymath}
S^{\deff}\to Y^{\deff} \to \mathcal{Y}
\end{displaymath}
where $S\to Y$ is dominant algebraic, and the latter map is a definable closed immersion.
\end{theor}
Since all this material is already covered in the write up of \cite{ben}, we omit the details that were discussed at CIRM. We conclude by quickly mentioning some of the more recent developments of the area.

\subsection{Recent work of  Bakker, Brunebarbe, and Tsimerman}
Another remarkable example is the Linear Shafarevich Conjecture in the quasi-projective case, by Bakker, Brunebarbe, and Tsimerman \cite{2024arXiv240816441B}, generalizing \cite{zbMATH06121648} that proved that a smooth \emph{projective} variety admitting an almost faithful representation of its fundamental group has holomorphically convex universal cover. The work of Bakker, Brunebarbe, and Tsimerman further uses the definable setting and among other things, the Ax-Schanuel theorem for abelian varieties. The authors prove that the universal cover of a normal complex algebraic variety admitting a faithful complex representation of its fundamental group is an analytic Zariski open subset of a holomorphically convex complex space.

In the process, they obtain the following:
\begin{theor}[{\cite[Thm. 1.7]{2024arXiv240816441B}}] Let $f: X\to Y$ be a proper morphism of seminormal definable analytic spaces. Then the Stein factorization of $f$ exists in the category of definable analytic spaces.
\end{theor}

\begin{rmk}
Definable GAGA and the above theorem give a new proof of the algebraicity of period maps. It is interesting to observe that this approach does not depend on the global definability of the period map, and only uses the local one (namely the nilpotent orbit theorem).
\end{rmk}

The paper in question is very long, but the definable parts closer to these lectures are mainly in \cite[Sec. 2 and 3]{2024arXiv240816441B}. (Admittedly the results described in this section appeared after the lecture series)

\subsection{Concluding remarks}
The viewpoint described before, as well as functional transcendence (at least at its debut) is tightly related with o-minimality. O-minimality has been applied with success to make progress on questions in Hodge theory (Griffiths conjecture, definable period maps \cite{gaga}), and has recently had its own explosion of results guided by Binyamini (sharply o-minimal sets, the resolution of Wilkie's conjecture, see e.g. \cite{2024arXiv240516963B}). 

Applications of the Zilber–Pink viewpoint have quickly permeated various areas and captured the attention of many mathematicians; to conclude these lectures we mention a few of the most recent ones. Study of the torsion locus of the Ceresa normal function by Gao and Zhang \cite{2024arXiv240701304G}, as well as related work by Kerr-Tayou \cite{2024arXiv240619366K} and Hain \cite{2024arXiv240807809H}, has offered further advancements. Description of the maximal compact subvarieties of Siegel modular varieties \cite{2024arXiv240406009G} —an area that has seen incremental progress over many years. A Diophantine direction, driven by the Lawrence and Venkatesh method \cite{zbMATH07233321}, with later contributions by Lawrence-Sawin; Javanpeykar-Krämer-Lehn-Maculan. A  Conjecture of Matsushita on Lagrangian fibrations of hyperk\"{a}hler manifolds \cite{2022arXiv220900604B, zbMATH07102431}. The work of Gao and collaborators on the uniform Mordell-Lang conjecture \cite{2021arXiv210403431G}. Further interesting studies of special loci associated to nontorsion admissible normal function \cite{arXiv:2407.10392}.

We expect many more exciting results to come!

\newpage
\section{Exercises}\label{lect5}
\subsection{Exercises for Lecture 1}
The main goal of this section is to introduce and give a feeling of the powerful Pila-Zannier strategy.
\subsubsection{Exercise 1.1}\label{ex11}
A longer exercise: justify the claims below.
We want to give a proof of the analogue of \Cref{althm1} for the $n$-dimensional algebraic torus:
\begin{displaymath}
\pi : \C^n \to (\C^*)^n.
\end{displaymath}
First of all: what are the \emph{(weakly) special} subvarieties in this case? Compare this notion with the one of \emph{bi-algebraic} subvarieties: i.e. algebraic subvarieties $Y \subset (\C^*)^n$ that are the projection along $\pi$ of an algebraic subvariety $Y'$ of $\C^n$.

The proof crucially uses the Pila-Wilkie counting theorem \cite{pilawilkie}, which was described in another set of lectures during the week at CIRM. Before starting with the guided exercise, we recall the counting theorem.
\begin{defi}[Algebraic part]
Let $Z \subset \R^n$. The algebraic part of $Z$, denoted by $Z^{\text{alg}}$ is the union of all connected, positive-dimensional semi-algebraic subsets of $Z$. 
\end{defi}

\begin{defi}[Counting function]
For a set $Z \subset \R^n$, an integer $k \geq 1$ and a real number $T \geq 1$, we define  
\begin{displaymath}
Z(k,T) := \{z = (z_1, \cdots, z_n) \in Z \colon \max_i [\Q(z_i) \colon \Q] \leq k \,\, , \max_{i}H(z_i) \leq T \}, 
\end{displaymath}
where $H$ denotes the absolute multiplicative height of an algebraic number. Then we set $$N(Z,k,T) := \# Z(k,T).$$
\end{defi}

\begin{theor}[Pila-Wilkie]
Let $Z \subset \R^n$ be definable in $\mathbb{R}_{\operatorname{an}, \exp}$. Let $k \geq 1$ an integer and $\epsilon > 0$. Then there is a constant $c(Z,k,\epsilon)$ such that for every $T \geq 1$ we have $$ N(Z - Z^{\text{alg}},k,T) \leq c(Z,k,\epsilon) T^{\epsilon}.$$
\end{theor}

We will now discuss a proof of the following:
\begin{theor}[Ax-Lindemann]
Let $V_1 \subset \C^n$, and $V_2 \subset (\C^*)^n$ algebraic. If $\pi (V_1)\subset V_2 $, then there exists $M \subset (\C^*)^n$ bi-algebraic such that 
\begin{displaymath}
\pi (V_1)\subset M \subset V_2.
\end{displaymath}
\end{theor}
\begin{rmk}
If the image of $\pi_1(V_2)$ is not finite index in $\pi_1((\C^*)^n)\cong \Z^n$, then $V_2$ is contained in a coset of a proper algebraic sub-torus.
\end{rmk}

The proof for the other inclusion is similar and it is also an exercise. Assume everything is irreducible, $V_2$ is not contained in any proper subtorus, $V_1$ is a maximal closed irreducible algebraic subvariety of $\pi^{-1}(V_2)$. 

Goal: $\text{Stab}_{\Z^n}(V_1)$ is infinite (it implies that $V_1$ is bi-algebraic).

Is $\pi$ definable? Can we restrict it to a smaller subset $\mathcal{F}\subset \C^n$ in order to make it definable and without loosing too much information? 

Consider the set:
\begin{displaymath}
I:=\{v \in \R^n : \dim (V_1 +v) \cap \pi^{-1}_{\mathcal{F}}(V_2)= \dim V_1\}.
\end{displaymath}
Notice that $v \in I$ if and only if $V_1 +v$ meets $\mathcal{F}$ and $V_1+v \subset\pi^{-1}(V_2)$.

\begin{enumerate}
\item $I$ is definable.
\item $V_1$ must pass through at least one fundamental domain $\mathcal{F}-v$ of each height.  So $N(I,t) \geq t+1$.
\item Apply Pila-Wilkie: there exists a real semialgebraic curve $C \subset I$ that contains at least two $\Z^n$-points.
\item $V_1 =V_1+c$ for any $c\in C$.
\item Argue by induction: $V_1$ is stabilised by a complex line $\C \subset \C^n = \C \oplus \C^{n-1}$ defined over $\Q$.
\end{enumerate}

\subsubsection{Exercise 1.2}\label{ex12}

Understand the Pila-Zannier strategy to prove the following.

\begin{theor}[Manin-Mumford for tori]
Let $G$ be the complex group $\mathbb{G}_m^n$ and $\Sigma \subset G$ be its torsion point. Let $V\subset G$ be an irreducible subvariety passing trough the identity $e\in G$. Then $V \cap \Sigma$ is Zariski dense in $V$ iff $V$ is an algebraic subgroup of $G$.
\end{theor}
We outline here six main steps/hints. 
\begin{enumerate}
\item Make the uniformizing map $\pi: \C^n \to G$ definable on some fundamental domain $\mathcal{F}$.
\item Lift torsion points of order $N$ to points of $\mathcal{F}$ of `small height'.
\item Can we apply the Pila-Wilkie theorem to $\pi^{-1}(V)$?
\item By using Galois conjugated of torsion points (and a bound on the growth of $[\Q(\mu_N):\Q]$), use step (3) to prove that $\pi^{-1}(V)$ has to contain a semialgebraic subset. 
\item Conclude if $V$ has dimension 1, by using the previous step.
\item An induction argument.
\end{enumerate}

\subsubsection{Exercise 1.3}\label{ex13}
 Formulate the Zilber-Pink conjecture for $G$, and check that it implies the above. (Originally due to Bombieri, Masser, and Zannier in `99 \cite{zbMATH01402564} and Zilber).
\subsubsection{Exercise 1.4}\label{ex14}
 Let $C \subset Y_0(1)\times Y_0(1)$. Understand what it means for a point $p=(x,y)\in C$ to be in the Hodge locus with respect to the natural VHS given by the $H^1$ of the elliptic curves. Which points are typical, which are atypical (with respect to $C$)? The answer to the latter will depend on properties of $C$.
\subsubsection{Exercise 1.5}\label{ex15}
Show that an analogue of \Cref{thmao} can fail to be true, if $V$ is not algebraic (but just a complex submanifold of $\C^2$).

\newpage
\subsection{Exercises for Lecture 2}\label{sectionex2}
The main goal of this section is to give concrete examples of Hodge loci associated to family of varieties as well as understand the relashonship between the various conjectures stated in the lecture. 
\subsubsection{Exercise 2.1}\label{ex21}
Prove that the various definitions of Mumford-Tate groups of a Hodge structure given in \Cref{lect2} are equivalent.

\subsubsection{Exercise 2.2}\label{ex22}
Find examples of pairs $(S,\VV)$ where $\HL(S,\VV)$ is just atypical, or just typical. Where the positive dimensional part of the Hodge locus is empty, and so on.
\subsubsection{Exercise 2.3}\label{ex23}
Let $(S,\VV)$ such that $\dim S=1$ and the generic Mumford-Tate is almost simple. Can the Hodge Locus of $(S,\VV)$ be infinite? Try to classify the possible Hodge data $(G,D)$ associated to the $(S,\VV)$ give rise to an infinite typical $\HL(S,\VV^{\otimes})$. What is the maxim of the dimensions of the strict sub-Shimura varieties of $\mathcal{A}_g$ (or the minimum of the codimensions)?
\subsubsection{Exercise 2.4}\label{ex24}
 Find examples of Shimura varieties (of dimension $>1$) without positive dimensional strict sub-Shimura varieties (and therefore without special subvarieties of dimension >0). Prove that CM points exist on any Shimura variety.
\subsubsection{Exercise 2.5}\label{ex25}
Some remarks on the endomorphism of Jacobians. Denote by $\mathcal{M}_g$ the moduli space of genus $g$ curves.

 For any totally real field $K$ of degree $g$, consider the set
\begin{displaymath}
I_K:=\{x\in \mathcal{M}_g : \End(Jx)\otimes \Q=K\},
\end{displaymath}
where $Jx$ denotes the Jacobian associated to the curve $x$. What do you expect, form the Zilber-Pink philosophy, on $I_K$? E.g. is it finite or infinite? The answer depends on $g$ (cf. \Cref{ex23}).

Given an integer $k$ between $1$ and $g -1$, is the set of genus $g$ curves whose Jacobian contains a $k$-dimensional abelian subvariety dense in $\mathcal{M}_g$? Find other examples of similar Hodge theoretic properties of Jacobians and understand if correspond to typical or atypical intersections.
\subsubsection{Exercise 2.6}\label{ex26}
 Show that the Griffiths transversality condition is necessary for Zilber-Pink for VHS to possibly hold true.
\subsubsection{Exercise 2.7}\label{ex27}
A longer exercise on the Noether-Lefchetz locus. Let
$U_d=\mathbb{P}H^0(\mathbb{P}^3, \Oo(d))-\Delta$ be the scheme parametrizing
smooth surfaces $X$ of degree $d$ in $\mathbb{P}^3$. Consider the so called
\emph{Noether-Lefschetz locus}: 
\begin{displaymath}
\NL_d:=\{[X]\in U_{2,d} : \operatorname{Pic}(\mathbb{P}^3)\to
\operatorname{Pic}(X) \text{  is not an isomorphism} \}. 
\end{displaymath}
Each $X$ outside $\NL_d$ has the following pleasant and useful
property: every curve on $X$ is the complete intersection of $X$ with
another surface in $\mathbb{P}^3$.  
\begin{itemize}
\item Try to prove Noether's theorem: $\NL_d$ is a countable union of strict
irreducible algebraic subsets of $U_d$. (Or disprove it, if $d$ is too small)
\item Prove that $\NL_d$ lies in the opportune Hodge locus.
\item Let $Y$ be a component of $\NL_d$ ($d\geq 4$). What do you expect about its codimension? More precisely, what can we say on $a,b$ such that
\begin{displaymath}
a \leq \codim_{U_d}Y \leq b?
\end{displaymath}
\item What happens if $d\geq 5$?
\end{itemize}

\subsubsection{Exercise 2.8}\label{ex28}
Should the Hodge locus come equipped with a non-reduced structure? Can this be achieved?
\subsubsection{Exercise 2.9}\label{ex29}
Does the Zilber-Pink conjecture stated in \Cref{lect2} implies all the results from \Cref{lect1}? Understand exactly the link between the ZP for VHS and the one for $\mathcal{A}_g$, semi-abelian varieties etc.
\subsubsection{Exercise 2.10}\label{ex210}
Show that \Cref{main conj} implies \Cref{conj-typical}.
\subsubsection{Exercise 2.11}\label{ex211}
Prove \Cref{hodgethm}. The main idea is to transform two typical intersections into an atypical one, by working in the product of the two associated period domains.
\subsubsection{Exercise 2.12}\label{ex212}
What does \Cref{main conj} predict about the subset of the Hodge locus of $(S,\VV)$ given by
\begin{displaymath}
\{s\in S(\C) : \operatorname{MT}(\VV_s) \text{  is  commutative} \}?
\end{displaymath}
The answer is the so called \emph{André-Oort conjecture} for VHS and includes \Cref{thmao} as a special case. A proof in level one has recently been obtained in \cite{2021arXiv210908788P}.

\newpage
\subsection{Exercises for Lecture 3}
In this section we will explore the basic properties of the weakly special subvarieties of $(S,\VV)$ and try to better understand the various Ax-Schanuel theorems that appeared so far.
\subsubsection{Exercise 3.1}\label{ex31}
Recall the fixed-part theorem and the semisimplicity theorems (e.g. \cite[Sec. 4]{MR0498551}) for VHS. Prove the following, due to Deligne and André \cite[Thm. 1]{MR1154159} (see also \cite{MR0498551}). Let $\VV$ be a $\Z$VHS on $S$. For any closed point $s\in S$, let $\mathbf{G}_s$ be the Mumford-Tate group of $\VV$ at $s$, and $\mathbf{H}_{\VV,s}$ the monodromy at $s$.
\begin{theor}[André-Deligne Monodromy Theorem]\label{monodromytheorem}
For a generic $s \in S$, the monodromy group $\mathbf{H}_{\VV,s}$ is a normal subgroup of the derived subgroup of $\mathbf{G}_s$.
\end{theor}
Hint: to prove that a subgroup $H\subset G \subset GL(V)$ is normal, it is enough to show that, for every tensor space $T^{m,n}$ and every $H$-character $\chi$, $(T^{m,n})^\chi$ is stable under $G$.
\subsubsection{Exercise 3.2}\label{ex32}
Prove that weakly-special subvarieties of $(S,\VV)$, as defined in the lecture, are the same as the bi-algebraic subvarieties (using the previous fact).  Prove that special subvarieties are weakly special. What is a weakly special point?

Prove that the same holds just assuming that a component of the preimage of $S $ in $D$ is definable (rather than algebraic), in some o-minimal structure expanding $\R_{\an}$.
\subsubsection{Exercise 3.3}
Understand the implications between the various statements of Ax-Schanuel type that appeared so far. For example, state the one for Shimura varieties, find an Ax-Lindemann type of statement in each setting (generalizing the one we used for tori). Prove that \Cref{thm:newAS} implies \Cref{asperiod}, and that the latter implies \Cref{astheorem}.
\subsubsection{Exercise 3.4}\label{ex35}
In this longer exercise we explain where functional transcendence come from, and why the main theorem is named after the mathematicians J. Ax and S. Schanuel.
Recall a result and a conjecture in transcendence theory:
\begin{itemize}
\item Let $z_1,\dots,z_n \in \Qbar$ be $\Q$-linearly independent,  then $\exp(z_i)_i$ are algebraically independent over $\Q$ (this is the so called Lindemann-Weierstrass theorem). That is $\text{tr.deg.}_\Q \Q(e^{z_1},\dots , e^{z_n})=n$.

\item Schanuel's conjecture gives a vast generalization of the above: Given $z_1, \dots ,z_n\in \C$ that are linearly independent over $\Q$,   the field extension
\begin{displaymath}
\Q (z_1,\dots , z_n, \exp{(z_1)},\dots, \exp{(z_n)})
\end{displaymath}
 has transcendence degree at least $n$ over $\Q$.  
\end{itemize}
To deal with Functional Transcendence, we just replace $\Q\subset \C$ by $\C \subset \C[[t_1,\dots, t_m]]$ and follow the beautiful work of Ax \cite{zbMATH03407823}.

\begin{theor}[Ax]
Let $x_1,\dots, x_n \in \C[[t_1,\dots, t_m]]$ have no constant term and be linearly independent over $\Q$. Then $\operatorname{tr.deg.}_\C \C(x_1,\dots, x_n, e^{x_1},\dots , e^{x_n})$ is at least $ n +\text{rank} \left(\frac{\partial x_i}{\partial t_j}\right)$.  
\end{theor}
Check that the above is equivalent to the following \emph{Geometric} formulation (same for AL that we discussed early on). Let $W\subset \C^n \times (\C^*)^n$ be an irreducible algebraic subvariety. Let $U$ be an irreducible analytic component of $W \cap \Pi$, where $\Pi$ is the graph of the exponentiation map.   Assume that the projection of $U$ to $(\C^*)^n$ is not contained in a translate of any proper algebraic subgroup. Then $\dim W = \dim U +n$. 
\subsubsection{Exercise 3.5}\label{ex37}
Prove \Cref{dense}, following the sketch given below.

\begin{proof}[Sketch of the proof]
Denote by $\tilde{S}$ the universal cover of $S$, and $\tilde{s}\in \tilde{S}$ a Hodge generic point. Fix $(\mathbf{H}, D_H) \subset (\mathbf{G}, D_G)$ a $\VV$-admissible Hodge sub-datum and $g\in \mathbf{G}(\R)$ such that $\tilde{s}\in g \cdot D_H$. Consider $\mathcal{ U}=\tilde{S}\cap g \cdot D_H \subset D_G$, which contains $\tilde{s}$. Since $\tilde{s}$ is Hodge generic, Ax-Schanuel (with the admissibility condition plugged in) implies that $\mathcal{ U}$ has an analytic irreducible component of the expected dimension (at $\tilde{s}$). The same holds true for any $g '$ sufficiently close to $g$, and we conclude by density of the rational points of $\mathbf{G}$ in the real ones.
 \end{proof}
 \subsubsection{Exercise 3.6}\label{ex3333}
 Check that the sets defined in \Cref{defsets} are definable. Verify all claims in \Cref{pro110}.
  \subsubsection{Exercise 3.7}\label{ex38} 
Deduce \Cref{thm00} from \Cref{pro110}. The idea is to show that the set $\{gMg^{-1}: (x,g,M)\in \Pi_2\}$ is definable and use the fact that there are at most countably many families of weakly special subvarieties to deduce the desired finiteness, from the very defining axiom of o-minimality. (See \cite[Prop. 6.6]{2021arXiv210708838B} for a similar argument).
 \subsubsection{Exercise 3.8}\label{ex39} 
 
 Deduce \Cref{geometricZP} from \Cref{thm00}. Here is where the families of weakly special subvarieties enter in the picture. A simpler intermediate step is to consider the case where the ambient monodromy group $\mathbf{H}$ is simple. Cf. \cite[End of Sec. 6.3]{2021arXiv210708838B} as well as \cite[Sec. 7.2]{2024arXiv240616628B}.
 
\subsubsection{Exercise 3.9}\label{ex399}
Using the ideas described in the previous exercises, prove \Cref{serrequestion} (it will follow from \Cref{ex37} and the André-Oort theorem for $\mathcal{A}_4$, proven for every $\mathcal{A}_g$ in \cite{MR3744855}).  
 
\newpage
\subsection{Exercises for Lecture 4}

\subsubsection{Exercise 4.1}
Prove the following, using the hints below (cf. \cite[Lem. 4.5]{sch73}):
\begin{theor}[Borel's monodromy theorem]
Let $\mathbb{V}\to \Delta^*$ be a polarized variation of pure $\Z$-Hodge structures of weight $k$ over the punctured disc $\Delta^*$, with period map $\Phi: \Delta^*\to \Gamma \backslash D$. The monodromy transformation of the local system $\mathbb{V}$ is quasi-unipotent.
\end{theor}
In more concrete theorems the theorem says the following. Since $\Phi \circ p$ is locally liftable and $\Hh$ is simply connected, we have a $\tilde{\Phi}: \Hh \to D$ making the diagram commutative
\begin{center}
\begin{tikzpicture}[scale=2]
\node (A) at (-1,1) {$\Hh$};
\node (B) at (1,1) {$D=G/M$};
\node (C) at (-1,0) {$\Delta^*$};
\node (D) at (1,0) {$\Gamma \backslash D$};
\path[->,font=\scriptsize,>=angle 90]
(A) edge node[above]{$\tilde{\Phi}$} (B)
(A) edge node[right]{$p=\exp (2 \pi i -)$} (C)
(B) edge node[right]{$\pi$} (D)
(C) edge node[above]{$\Phi$} (D);
\end{tikzpicture}.
\end{center}
The fundamental group of $\Delta ^*$, seen as a transformation group of $\Hh$, is generated by translation $z\mapsto z+1$. We can choose a $\gamma \in \Gamma \subset G(\Z)\subset \Gl_n (V_\Z)$ such that
\begin{equation}
\forall z \in \Hh \ \ \ \tilde{\Phi}(z+1)=\gamma \tilde{\Phi}(z).
\end{equation}
(in our case $\gamma$ is just the image of a generator of $\pi_1(\Delta^*)$ in $\Gamma$). The theorem then says that the eigenvalues of $\gamma$ are roots of unity and a suitable power of $\gamma$ is nilpotent.

\begin{itemize}
\item Step 1: it is enough to prove that the eigenvalues of $\gamma$ have absolute value one. In fact we want to show  that the conjugacy class of $\gamma$ in $G_\R$ has an accumulation point in the compact subgroup $M \subset G_\R$.
\item Step 2: Consider the points $\{\tilde{\Phi} (i \cdot n)  = g_n M \}_{n \in \N}$ (for some $g_n \in G_\R$). We claim that the sequence $\{g_n^{-1}\gamma g_n\}$ converges to the compact subgroup $M$. 
\item  Let $d$ be a $G_\R$ invariant Riemannian distance on $D$ (suitably renormalised), prove that 
\begin{displaymath}
d (g_n^{-1}\gamma g_n M,M) = d (\gamma g_n M,g_n M)  \leq 1/n.
\end{displaymath}
(use the fact that the coset $\gamma g_n M$ corresponds to $\tilde{\Phi} (in+1)$ and that $\tilde{\Phi}$ does not increase distances...).
\end{itemize}

\subsubsection{Exercise 4.2}
Let $S\subset \overline{S}$ be a smooth compactification such that $\overline{S}-S$ is a normal crossing divisor. Let $\overline{S}_i$ be a finite open cover of $\overline{S}$ such that 
\begin{displaymath}
(\overline{S}_i, S_i:=S_i \cap \overline{S}_i) \text{ is biholomorphic to  } (\Delta^n, (\Delta^*)^r_i\times \Delta^{n-r_i}). 
\end{displaymath}
To show that $\Phi : S \to S_{\Gamma , G, M}$ is $\R_{\an, \exp}$-definable, it is enough to prove that the restriction of $\Phi$ to each $S_i$ is definable. We may also assume that $r_i=n$. The goal of this exercise is to review the proof of the following, since it is the starting point of Lecture 4, and was used in Lecture 3 as well.
\begin{theor}[Bakker-Klingler-Tsimerman]
Let $\mathbb{V}\to (\Delta^*)^n$ be a VHS (of some weight $k$), with period map $\Phi :(\Delta^*)^n\to S_{\Gamma , G,M}$. Then $\Phi $ is $\R_{\an, \exp}$-definable.
\end{theor}

More generally let $\mathbb{V}$ be a local system over $(\Delta^*)^n$ and $T_i\in G(\Z)$ the monodromy transformations (=counterclockwise simple circuits around the $n$-punctures). Up to an étale covering we may assume $T_i$ to be unipotent. Let $N_i \in \mathfrak{g}_\Q$ their logarithms: they are nilpotent elements in $\mathfrak{g}_\Q$. Let $\F$ be the Siegel fundamental set for $\Hh$:
\begin{displaymath}
\F:=\{y>1 , -1/2 <x<1/2\} \subset \Hh.
\end{displaymath}
By def, the restriction of $\exp (2\pi i \cdot)$ to $\F$ is definable, so we are reduced to proving that the composition
\begin{displaymath}
\F^n \to G/M \to S_{\Gamma, G, M}
\end{displaymath}
is $\R_{\exp , an}$-definable. For the first map we have

\begin{itemize}
\item Prove that  map $\tilde{ \Phi}: \F^n \to G/M$ is $\R_{\an,\exp}$-definable, by using the nilpotent orbit theorem.

\item What is missing to complete the proof of th theorem?
\end{itemize}

A crucial input in the theory is the so called $\Sl_2^n$-\emph{orbit theorem} \cite{sch73}:
\begin{theor}[Schmid]

Let $\Phi : (\Delta^*)^n\to S_{\Gamma, G,M}$ be a period map with unipotent monodromy. Let $\tilde{\Phi}: \Hh \to G/M$ be its lifting. For any given constants $R,\eta >0$ there 
exists finitely many Siegel sets $\mathcal{G}_i \subset G/M$ such that $\Phi (z)\in \bigcup_i \mathcal{G}_i$ whenever $\operatorname{Re}(z)\leq C$ and $\operatorname{Im} (z)\geq \eta$.
\end{theor}

\subsubsection{Exercise 4.3}
Consider a family of algebraic surfaces of degree at least three. Then $X/\mathbb{P}^1$ has at least three singular fibers. 

\subsubsection{Exercise 4.4}

The goal of this exercise is to prove, using o-minimality and the previous results, Borel’s algebraicity theorem. Let $\Gamma \backslash D$ be a Mumford-Tate domain. We discussed examples where  $\Gamma \backslash D$ has an algebraic structure. Convince yourself that this is not always the case. Prove that, it can have at most one algebraic structure.

Prove the following (cf. \cite[Thm. 4.12]{BKT}):
\begin{theor}
Let $S$ be a complex algebraic variety and $\Gamma \backslash D$ be an algebraic Mumford-Tate domain (or a Shimura variety). Every $f: S(\C) \to \Gamma \backslash D(\C)$ complex analytic map is the analytification of an algebraic map $S\to \Gamma \backslash D$.
\end{theor}

\newpage

\bibliographystyle{abbrv}

\bibliography{biblio.bib}

\Addresses

\end{document}